\documentclass[a4paper,reqno,12pt]{article}

\usepackage{latexsym,amssymb, amsthm, epsfig}
\usepackage{helvet}         
\usepackage{courier}        
\usepackage{multicol}

\usepackage{color}
\usepackage{graphicx}
\usepackage{amsmath}
\usepackage{amsfonts}

\newtheorem{theorem}{\bf Theorem}[section]
\newtheorem{lemma}[theorem]{\bf Lemma}
\newtheorem{corollary}[theorem]{\bf Corollary}
\newtheorem{proposition}[theorem]{\bf Proposition}
\newtheorem{remark}{\bf Remark}[section]
\newtheorem{definition}{\bf Definition}[section]

\newcounter{for}[section]

\newcommand{\be}[1]{\addtocounter{for}{1}\begin{equation}\label{#1}}
\newcommand{\ee}{\end{equation}}

\def\R{{\mathbb R}}

\def\N{{\mathbb N}}

\def\cM{{\mathcal{M}}}
\def\cN{{\mathcal{N}}}

\def\ka6{6}

\def\ve{{\varepsilon}}

\def\({{\Bigl(}}
\def\){{\Bigr)}}
\def\reff#1{(\ref{#1})}
\def\one{{\mathbf 1}}
\def\ind{{\mathbf 1}}
\def\square{\ifmmode\sqr\else{$\sqr$}\fi}
\def\sqr{\vcenter{
         \hrule height.1mm
         \hbox{\vrule width.1mm height2.2mm\kern2.18mm\vrule width.1mm}
         \hrule height.1mm}}                  

\newcommand {\cro}[1] {\left[ {#1} \right]}
\newcommand {\acc}[1] {\left\{ {#1} \right\}}
\newcommand {\pare}[1] {\left( {#1} \right)}

\theoremstyle{plain}

\theoremstyle{definition}

\theoremstyle{remark}

\def\epr{\end{proof}}

\def\bpr{\begin{proof}}

\def \beq {\begin{eqnarray}}
\def \eeq {\end{eqnarray}}
\def \beqn {\begin{eqnarray*}}
\def \eeqn {\end{eqnarray*}}

\newcommand{\bl}[1]{\begin{lemma}\label{#1}}
\newcommand{\br}[1]{\begin{remark}\label{#1}}
\newcommand{\brs}[1]{\begin{remarks}\label{#1}}
\newcommand{\bt}[1]{\begin{theorem}\label{#1}}
\newcommand{\bd}[1]{\begin{definition}\label{#1}}
\newcommand{\bp}[1]{\begin{proposition}\label{#1}}
\newcommand{\bc}[1]{\begin{corollary}\label{#1}}
\newcommand{\bfact}[1]{\begin{fact}\label{#1}}
\newcommand{\bex}[1]{\begin{example}\label{#1}}
\newcommand{\ec}{\end{corollary}}
\newcommand{\efact}{\end{fact}}
\newcommand{\eex}{\end{example}}
\newcommand{\el}{\end{lemma}}
\newcommand{\er}{\end{remark}}
\newcommand{\ers}{\end{remarks}}
\newcommand{\et}{\end{theorem}}
\newcommand{\ed}{\end{definition}}

\newcommand{\ep}{\end{proposition}}

\newcommand{\bcl}[1]{\begin{claim}\label{#1}}
\newcommand{\ecl}{\end{claim}}

\newcommand{\ecs}{\end{corollary}}
\newcommand{\eers}{\end{exercise}}
\newcommand{\eexs}{\end{example}}
\newcommand{\eems}{\end{example}}
\newcommand{\els}{\end{lemma}}
\newcommand{\eles}{\end{lemmaex}}
\newcommand{\ets}{\end{theorem}}
\newcommand{\eds}{\end{definition}}
\newcommand{\eps}{\end{proposition}}

\newcommand{\bi}{\begin{itemize}}
\newcommand{\ei}{\end{itemize}}
\newcommand{\ben}{\begin{enumerate}}
\newcommand{\een}{\end{enumerate}}

\def\sqr#1#2{{\vcenter{\vbox{\hrule height .#2pt
                             \hbox{\vrule width .#2pt height#1pt \kern#1pt
                                   \vrule width .#2pt}
                             \hrule height .#2pt}}}}
\def\square{\mathchoice\sqr54\sqr54\sqr{4.1}3\sqr{3.5}3}
\def\pmb#1{\setbox0=\hbox{#1}%
   \kern-.025em\copy0\kern-\wd0
   \kern.05em\copy0\kern-\wd0
   \kern-.025em\raise.0433em\box0 }
\def\sqr#1#2{{\vcenter{\vbox{\hrule height.#2pt
     \hbox{\vrule width.#2pt height#1pt \kern#1pt
   \vrule width.#2pt}\hrule height.#2pt}}}}

\def\R{{\mathbb R}}

\def\cG{\mathcal G}

\def\reff#1{(\ref{#1})}

\def\qsd{{\sc{qsd}}}
\def\fv{{Fleming-Viot}}
\def\gw{{Galton-Watson}}

\def\nuqs{{\nu_{\rm qs}^*}}
\def\rmost{{R}}
\def\lambdan{{\lambda^{\text{\hskip-.7mm\tiny\it{N}}}}}
\def\psil{{\psi^{\text{\hskip-.3mm\tiny\it{L}}}}}
\def\cLn{\mathcal L^{\text{\hskip-.4mm\tiny\it{N}}}}

\def\bF{\mathbf F}

\def\tQg{{\widetilde Q}}

\def\tZ{{\widetilde Z}}
\def\tSt{{\widetilde S_t}}

\begin{document}

\title{Fleming-Viot selects the minimal quasi-stationary distribution: 
The Galton-Watson case.}

\author{Amine Asselah, Pablo A. Ferrari, Pablo Groisman, Matthieu Jonckheere
\\ 
\sl Universit\'e Paris-Est, Universidad de Buenos Aires, IMAS-Conicet}
\date{}
%
%
%


%
%
%

\maketitle

\noindent {\bf Abstract}
  Consider $N$ particles moving independently, each one according to
  a subcritical continuous-time Galton-Watson process unless it hits
  $0$, at which time it jumps instantaneously to the position of one
  of the other particles chosen uniformly at random. The resulting
  dynamics is called Fleming-Viot process.  We show that for
  each $N$ there exists a unique invariant measure for the
  Fleming-Viot process, and that its stationary empirical distribution
  converges, as $N$ goes to infinity, to the minimal
  quasi-stationary distribution of the Galton-Watson process
  conditioned on non-extinction. 

\vskip 3mm

\noindent {{\it AMS 2000 subject classifications}.  Primary 60K35;
Secondary 60J25}

\noindent {\it Key words and phrases}.  Quasi-stationary distributions,
Fleming-Viot processes, branching processes, selection principle.

%


\section{Introduction}
The concept of {\it quasi-stationarity} arises in stochastic modeling of
population dynamics. In 1947, Yaglom \cite{yaglom} considers subcritical
Galton-Watson processes conditioned to survive long times.  He shows that as
time is sent to infinity, the conditioned process, started with one individual,
converges to a law, now called a {\it quasi-stationary distribution}.  For any
Markov process, and a subset $A$ of the state space, we denote by $\mu T_t$ the
law of the process at time $t$ conditioned on not having hit $A$ up to time $t$,
with initial distribution $\mu$. A probability measure on $A^c$ is called
\emph{quasi-stationary distribution} if it is a fixed point of $T_t$ for any
$t>0$.

In 1966, Seneta and Veres-Jones \cite{seneta-veres} realize that for subcritical
Galton-Watson processes, there is a one-parameter family of quasi-stationary
distributions and show that the Yaglom limit distribution has the minimal
expected time of extinction among all quasi-stationary distributions. This
unique {\it minimal} quasi-stationary distribution is denoted here $\nuqs$. They
also show that with an initial distribution $\mu$ with finite first moment, $\mu
T_t$ converges to $\nuqs$ as $t$ goes to infinity.

In 1978, Cavender \cite{cavender} shows that for Birth and Death chains on the
non negative integers absorbed at 0, the set of quasi-stationary measures is
either empty or is a one parameter family. In the latter case, Cavender extends
the {\it selection principle} of Seneta and Veres-Jones.  He also shows that the
limit of the sequence of quasi-stationary distributions for truncated processes
on $\{1,\dots,L\}$ converges to $\nuqs$ as $L$ is sent to infinity.  This
picture holds for a class of irreducible Markov processes on the non-negative
integers with 0 as absorbing state, as shown in 1996 by Ferrari, Kesten,
Martinez and Picco \cite{ferrari-kesten}.  The main idea in
\cite{ferrari-kesten} is to think of the conditioned process $\mu T_t$ as a mass
transport with refeeding from the absorbing state to each of the transient
states with a rate proportional to the transient state mass.  More precisely,
denoting $\N$ the set of positive integers, the Kolmogorov forward equation
satisfied by $\mu T_t(x)$, for each $x\in\N$, reads
 \be{kfe}
\frac{\partial}{\partial t} \mu T_t(x) = \sum_{y: y\ne x}\big(q(x,y) + q(x,0)
\mu T_t(y)\big)\,[\mu T_t(y)- \mu T_t(x)],  \ee
where $q(x,y)$ is the jump rate from $x$ to $y$. The first term in the right
hand side represents the displacement of mass due to the jumps of the process
and the second term represents the mass going from each $x$ to 0 and then coming
instantaneously to $y$.

In 1996, Burdzy, Holyst, Ingerman and March \cite{burdzy1} introduced a {\it
  genetic} particle system called \emph{Fleming-Viot} named after models
proposed in \cite{fv}, which can be seen as a particle system mimicking the
evolution \eqref{kfe}. The particle system can be built from a process with absorption
$Z_t$ called \emph{driving process}; the position $Z_t$ is interpreted as
a genetic trait, or fitness, of an individual at time $t$.  In the $N$-particle
Fleming-Viot system, each trait follows independent dynamics with the same law
as $Z_t$ except when one of them hits state 0, a lethal trait: at this moment
the individual adopts the trait of one of the other individuals chosen uniformly
at random.  Leaving aside the genetic interpretation, the empirical distribution
of the $N$ particles at positions $\xi\in\N^N$ is defined as a function
$m(\cdot,\xi): \N\to[0,1]$ by
\be{empiric} \forall x\in \N,\qquad m(x,\xi) :=
\frac1N\sum_{i=1}^N \ind_{\{\xi(i)=x\}}.  \ee 
The generator of the Fleming-Viot process with $N$ particles applied to bounded
functions $f:\N^N\to\R$ reads
\be{generatorfv}
\cLn f(\xi)= \sum_{i=1}^N \sum_{y=1}^\infty \Bigl[q(\xi(i),y)+ q(\xi(i),0)\,
\textstyle{\frac{N}{N-1}} \,m(y,\xi)\Bigr]\, [f(\xi^{i,y})-f(\xi)],
\ee
where $\xi^{i,y}(i)=y$, and for $j\not=i$, $\xi^{i,y}(j) =
\xi(j)$ and $q(x,y)$ are the jump rates of the driving process.  Assume that the
driving process has a unique quasi-stationary distribution, called $\nu_{\rm
  qs}$ and that the associated $N$-particle Fleming-Viot system has an invariant
measure $\lambdan$. The main conjecture in \cite{burdzy1,burdzy2} is that
assuming $\xi$ has distribution $\lambdan$, the law of the random measure
$m(.,\xi)$ converges to the law concentrated on the constant $\nu_{\rm qs}$.
This was proven for diffusion processes on a bounded domain of~$\R^d$, killed at
the boundary~\cite{bieniek, GK1, GK,villemonais}, for jump
processes under a Doeblin condition~\cite{ferrari-maric} and for finite state
jump processes~\cite{AFG}.

The subcritical \gw{} process has
infinitely many quasi-stationary distributions.  Our theorem proves that
the stationary empirical distribution $m(\cdot,\xi)$ converges to $\nu_{\rm qs}^*$,
the minimal quasi-stationary distribution.  This phenomenon is a \emph{selection
  principle}.

\begin{theorem}\label{theo-main}
  Consider a subcritical Galton-Watson process whose offspring law has some
  finite positive exponential moment.  Let $\nuqs$ be the minimal 
quasi-stationary distribution for the process conditioned on non-extinction.
Then, for each $N\ge 1$, the associated  $N$-particle Fleming-Viot system
is ergodic.
Furthermore, if we call its invariant measure $\lambdan$, then
\be{main-1}
\forall x\in \N,\qquad
\lim_{N\to\infty}\int |m(x,\xi) - \nuqs(x)| \, d\lambdan(\xi) \;=\;0.
\ee
\end{theorem}
A simple
consequence is propagation of chaos. For any finite set $S \subset \N$,
\be{limit-chaos} \lim_{N \to \infty} \int\prod_{x \in S} m(x,\xi) \
d\lambdan(\xi)= \prod_{x \in S} \nuqs(x).  \ee

The strategy for proving Theorem~\ref{theo-main} is explained in the next
section, but there are two key steps in the proof. First, we control the
position of the rightmost particle. Let
\[
\rmost(\xi)\;:=\; \max_{i\in\{1,\dots,N\}} \xi(i),
\]
be the position of the rightmost particle of $\xi$. Let $\xi^\xi_t$ the
positions at time $t$ of the $N$ Fleming-Viot particles, initially on $\xi$.

\bp{prop-rightmost} There is a time $T$ and positive
constants $A, c_1, c_2,C$ and
$\rho$, independent of $N$, such that for any
$\xi\in\N^N$ 
\be{foster}
 E\big(\exp\big(\rho\rmost(\xi^\xi_T)\big)\big) -\exp\big(\rho
\rmost(\xi)\big)\;<\;\;
-\,c_1\, e^{\rho\rmost(\xi)}\ind_{\rmost(\xi)>A}  + N c_2
e^{-C\rmost(\xi)}.
\ee
As a consequence, for each $N$ there is a unique invariant 
measure $\lambdan$ for the $N$-particle Fleming-Viot system.
Furthermore, there is a constant $\kappa>0$ such that for any $N$,
\be{rightmost-expon} 
\int \exp(\rho \rmost(\xi)) d\lambdan(\xi) \;\le\; \kappa N.
\ee
\ep 
The second result is that the ratio between the second and the
first moment of the empirical distribution plays the role of a
Lyapunov functional, given that the position of the 
rightmost particle is not
too large. For a particle configuration $\xi$ define
 \be{def-psi2} \psi(\xi):= \frac{\sum_{1\le i\le N}\xi^2(i)}
{\sum_{1\le i\le N} \xi(i)}.  \ee
Recall $\cLn$ is the Fleming-Viot generator given by \eqref{generatorfv}.

\bp{prop-key}  There are positive constants $v,C_1$ and $C_2$ independent of
$N$ such that
 \be{L-psi2} \cLn\psi(\xi)\;\le\; -v\psi(\xi)+C_1
\frac{\rmost^2(\xi)}{N}+C_2.  \ee \ep 
Propositions \ref{prop-rightmost} and \ref{prop-key}
imply that the expectation of $\psi$ 
under the invariant measure $\lambdan$ is
uniformly bounded in $N$.
\bc{cor-bary} There is a positive
constant $C$ such that for all $N$,
\be{ineq-bary} \int \psi(\xi)\, d\lambdan(\xi)
\;\le\; C.  \ee
\ec

There are several related works motivated by genetics. Brunet,
Derrida, Mueller and Munier \cite{BDMM1,BDMM2} 
introduce a model of evolution of a population with selection.
They study the genealogy of genetic traits, 
the empirical measure, and link the evolution of the barycenter
with F-KPP equation $\partial_t u=\partial_x^2 u-u(1-u)$ 
introduced in 1937 by R.A. Fisher to describe the evolution
of an advantageous gene in a population. These authors also
discover an exactly soluble model whose genealogy is
identical to those predicted by Parisi's theory of mean-field
spin glasses.
Durrett and Remenik \cite{durret} establish propagation of chaos for
a related continuous-space and time model, and then show
that the limit of the empirical
measure is characterized as the solution of a free-boundary
integro-differential equation. B\'erard and Gou\'er\'e
\cite{berard-gouere} establish a conjecture
of Brunet and Derrida for the speed of the rightmost particle
for still a third microscopic model of F-KPP equation
introduced in \cite{brunet-derrida1,brunet-derrida2}.  
Maillard \cite{maillard} obtains the precise behavior of
the empirical measure of an approximation of the same model, building
on the results of Berestycki, Berestycki and Schweinsberg
\cite{berestycki}, which establish the genealogy picture described in 
{\cite{brunet-derrida1,brunet-derrida2}}.

We now mention two open problems. The first
is to solve the analogous to Theorem~\ref{theo-main} for a random
walk with a constant drift toward the origin. The second is
to obtain propagation of chaos directly on the
stationary empirical measure, with a bound of order $1/N$.

In the next section, we describe our model, sketch the proof of our main result
and describe the organization of the paper.

\section{Notation and Strategy}\label{section2}
Let $\sigma>0$ and $p$ be a probability distribution on $\N\cup\{0\}$
such that 
\be{a15}
  \sum_{\ell\ge 0} p(\ell)\,e^{\sigma\ell}\;<\;\infty.
\ee
Consider a Galton-Watson process $Z_t\in\N\cup\{0\}$ with offspring law
$p$. Each individual lives an exponential time of parameter 1, and then gives
birth to a random number of children with law~$p$.
We assume that \gw{} is subcritical, that is we ask $p$ to satisfy
\be{def-drift} -v\;:=\;\sum_{\ell \ge -1} \ell p(\ell+1) \;<\;0.  \ee 
In other words, the drift when $Z_t=x$ is $-vx<0$.  For distinct $x,y\in\N\cup\{0\}$, the
rates of jump are given by 
\be{rates} q(x,y) := 
\left\{\begin{array}{ll}
    x p(0), &\text{if } y= x-1\ge 0,\\
    x p(y-x+1), &\text{if } y> x\ge1,\\
    0,&\text{otherwise}.
\end{array} \right.
\ee
The \gw{} process starting at $x$ is denoted $Z^x_t$. For a distribution $\mu$
on $\N$, the law of the process starting with $\mu$ conditioned on
non-absorption until time $t$ is given by
\be{muTt}
\mu T_t(y) := \frac{\sum_{x\in \N} \mu(x) p_t(x,y)}{\sum_{x,z\in \N} \mu(x) p_t(x,z)},
\ee
where $p_t(x,y)= P(Z^x_t=y)$.

Recall that $\xi^\xi_t$ denotes the \fv{} system with generator
\eqref{generatorfv} and initial state $\xi$; $\xi_t(i)$
denotes the position of the $i$-th particle at time $t$. 
For a real $\alpha>0$ define $K(\alpha)$ as the subset of
distributions on $\N$ given by
\be{def-K}
K(\alpha):=\acc{\mu:\ \frac{\sum_{x\in\N} x^2\mu(x)}{
\sum_{x\in\N} x\mu(x)}\le\alpha}.
\ee
Observe that $\mu \in K(\alpha)$ implies $\sum x\mu(x) \le \alpha$.
\begin{proof}[Proof of Theorem \ref{theo-main}]
The existence of the unique invariant measure $\lambdan$ for \fv{} 
is given in Proposition~\ref{prop-rightmost}.

To show \eqref{main-1} we use the invariance of $\lambdan$
and perform the following decomposition. 
\be{strategy-1}
\begin{split}
\int &\big| m(x, \xi) - \nuqs(x)\big| \, d\lambdan(\xi) 
\;= \;\int E\big|  m(x,\xi^\xi_t) -   \nuqs(x)\big| \, d\lambdan(\xi)\\
&\le\; \lambdan({\psi>\alpha})+ 
\int_{\psi \le \alpha}\!\! E\big|m(x,\xi^\xi_t) -   \nuqs(x)\big|
\, d\lambdan(\xi)\\
&\le\; \lambdan({\psi>\alpha})+ \int_{\psi\le\alpha} \!\!\!
E\big|m(x,\xi^\xi_t) -  m(\cdot, \xi)T_t(x)\big|\,d\lambdan(\xi) + 
\int_{\psi\le\alpha} \!\!\!\big|m(\cdot, \xi)T_t(x) - \nuqs(x)\big|
\, d\lambdan(\xi)\\
 &\le\;  \lambdan({\psi>\alpha})+
\sup_{\xi:\psi(\xi)\le\alpha}\!\! \big| m(\cdot,\xi)T_t (x) - \nuqs(x) \big|\,+\,
\sup_{\xi:\psi(\xi)\le\alpha}\!\! E\big|m(x,\xi^\xi_t)-m(\cdot,\xi)T_t (x)\big|,
\end{split}
\ee
where $\psi$ is defined in \eqref{def-psi2}.
We bound the three terms of the last line of \eqref{strategy-1}.

\noindent\emph{First term. } 
Corollary \ref{cor-bary} and Markov inequality imply that there is a constant
$C>0$ such that for 
any $\alpha>0$  
\be{I3-estimate} \lambdan(\psi>\alpha) \le
\frac{C}{\alpha}.  \ee
\emph{Second term. }  Note that $\psi(\xi)<\alpha$ if and only if
$m(\cdot,\xi)\in K(\alpha)$.  The Yaglom limit converges to the minimal
quasi-stationary distribution
$\nuqs$, uniformly in $K(\alpha)$ as we show later in Proposition
\ref{prop-yaglom}:
\be{unif.Yaglom}
\lim_{t \to \infty} \sup_{\mu\in K(\alpha)} | \mu T_t(x) - \nuqs(x) |=0.
\ee

\noindent\emph{Third term. } We perform the decomposition
\be{ad1}
E\big| m(x,\xi^\xi_t)- m(\cdot,\xi)T_t(x)\big|\;\le\;
E\big| m(x,\xi^\xi_t) -E
m(x,\xi^\xi_t)\big|\,+\, 
\big| Em(x,\xi^\xi_t)- m(\cdot,\xi)T_t(x)\big|,
\ee 
and show that there exist positive constants $C_1$ and $C_2$ 
such that
\be{var1} 
\sup_{\xi\in\N^N} E\big| m(x,\xi^\xi_t) -E m(x,\xi^\xi_t)\big|\;\le\;
\frac{C_1 e^{C_2 t}}{\sqrt N} 
\ee
and 
\be{var2}
\sup_{\xi\in\N^N} \big| Em(x,\xi^\xi_t)- m(\cdot,\xi)T_t(x)\big|\;\le\;
\frac{ C_1 e^{C_2 t}}{N}, 
\ee
for all $N$, see Proposition \ref{prop-semigroups} later. The issue here is a
uniform bound for the correlations of the empirical distribution of \fv{} at
sites $x, y \in \N$ at fixed time $t$.  This was carried out in \cite{AFG}.

To show \eqref{main-1}, it suffices to bound the three terms in the bottom line
of \eqref{strategy-1}. Choose $\alpha$ large and use \eqref{I3-estimate} to make
the first term small (uniform in $N$). Use \eqref{unif.Yaglom} to choose $t$
large to make the second term small. For this fixed time, take $N$ large and use
\eqref{ad1}, \eqref{var1} and \eqref{var2} to make the third term small.
\end{proof}

The rest of the paper is organized as follows.  In Section
\ref{sec-construction}, we perform the graphical construction of \fv{} jointly
with a Multitype Branching Markov Chain. In Section~\ref{Galton-Watson
  estimates} we obtain large deviation estimates for the Galton-Watson
process. In Section~\ref{rightmost-particle} we obtain large deviation estimates
for the rightmost particle of the \fv{} system.  In Section~\ref{moments-fv} we
study the Lyapunov-like functional and prove Proposition~\ref{prop-key} and
Corollary~\ref{cor-bary}.  Convergence of the conditional evolution uniformly on
$K(\alpha)$ is proved in Section~\ref{sec-yaglom}. Finally,
\eqref{var1}-\eqref{var2} are handled in 
Proposition~\ref{prop-semigroups} of Section~\ref{sec-semigroups}.

\section{Embedding \fv{} on a multitype branching
Markov process.}\label{sec-construction} 

In this section we construct a coupling between the Fleming-Viot system and an
auxiliary multitype branching Markov process (hereafter, the branching
process). We call \emph{particles} the \fv{} positions and \emph{individuals}
the branching positions. Each individual has a a type in $\{1,\dots,N\}$ and a
position in $\N$.

When a particle chooses the position of another
particle and jumps to it, the process builds correlations making difficult to control the
position of the rightmost particle. In our coupling when a particle jumps,
either an individual jumps at the same time or a branching occurs at the site
where the particle arrives. In this way the particles always stay at sites
occupied by individuals and the maximum particle position is dominated by the
position of the rightmost individual (if this is so at time zero). This, in turn,
is dominated by the sum of the individual positions which we control.

The coupling relies on the Harris construction of Markov processes: the state of
the process at time $t$ is defined as a function of the initial configuration
and a family of independent Poisson processes in the time interval $[0,t]$.  The
coupling holds when the driving process is a Markov process with rates
$\{q(x,y),x,y\in \N\cup\{0\}\}$ with $0$ being the absorbing state and $\bar
q:=\sup_x q(x,0)<\infty$.

There are two types of jumps of the Fleming-Viot particle $i$. Those due to the
spatial evolution at rate $\tilde q$ and those due to ``jumps to zero and then
to the position of particle $j$ chosen uniformly at random'' at rate
$q(x,0)/(N-1)$.
 
\emph{Spatial evolution. } Each individual has a position in $\N$ which evolves
independently with transition rates $(\tilde q(x,y), x,y \in \N)$ defined by
$\tilde q(x,y):=q(x,y)\ind_{\{y \ne 0 \}}$ so that there are no jumps to zero.
The spatial evolution of new individuals born at branching times are independent
and with the same rates $\tilde q$. Under our coupling, each spatial jump
performed by the
$i$-particle is also performed by some $i$-individual.

\emph{The refeeding and branching. }  At rate $\bar q/(N-1)$, each
$j$-individual branches into two new individuals, one of type $j$ and one of
type $i$; each new born $i$-individual takes the position of the corresponding
$j$-individual and then evolves independently with rates $\tilde q$. If the
$i$-particle is at $x$, at rate $q(x,0)/(N-1)$ it jumps to the position of the
$j$-particle. Under our coupling, each time particle $i$ chooses particle $j$,
each $j$-individual branches into an $i$ and a $j$-individual.  In this way, the
$i$-particle occupies always the site of some $i$-individual.

The  branching process has state space
\[
{\mathcal B}:= \Bigl\{\zeta\in \N^{\{1,\dots,N\}\times\N}: \sum_{i=1}^N\sum_{x\in\N}
\zeta(i,x)<\infty\Bigr\}
\]
For 
 $i\in\{1,\dots,N\}$, $x\in \N$, $\zeta_t(i,x)$ indicates the number of
 individuals of type $i$ at site $x$ at time $t$.
 Let $\delta_{(i,x)}\in{\mathcal B}$ be the delta function on $(i,x)$ defined by
 $\delta_{(i,x)}(i,x)=1$ and $\delta_{(i,x)}(j,y)=0$ for $(j,y)\neq(i,x)$. The
 rates corresponding to the (independent) spatial evolution of the individuals
 at $x$ are
\[
b(\zeta,\zeta + \delta_{(i,y)}-\delta_{(i,x)}) = \zeta(i,x)q(x,y),\qquad
i\in\{1,\dots,N\},\; x,y\in\N, 
\]
and those corresponding to the branching of all $j$-individuals into an
individual of type $j$ and an individual of type $i$ are 
\[
b\Bigl(\zeta,\zeta+ \sum_{x\in\N} \zeta(j,x)\delta_{(i,x)}\Bigr) = \frac{\bar q}{N-1}
,\qquad i\ne j\in\{1,\dots,N\}
\]
Note that the new born
$i$-individuals get the spatial position of the corresponding $j$-individual.

\noindent{\bf Harris construction of the branching process } 
Let $(\cN(i,x,y,k), \,i\in\{1,\dots,N\}, x,y\in \N, k\in \N)$ be a family of
Poisson processes with rates $k\tilde q(x,y)$ such that
$\cN(i,x,y,k)\subset\cN(i,x,y,k+1)$ for all $k$; we think a Poisson process as a
random subset of $\R$. The process $\cN(i,x,y,k)$ is used to produce a jump of
an $i$-individual from $x$ to $y$ when there are $k$ $i$-individuals at site
$x$. The families $(\cN(i,x,y,k),k\ge 1)$ are taken independent. Let
$(\cN(i,j)$, $i\ne j)$, be a family of independent Poisson processes of rate
$\bar q/(N-1)$, these processes are used to branch all $j$-individuals
into an $i$-individual and a $j$-individual. The two families are taken
independent.

Fix $\zeta_0=\zeta\in\mathcal B$, assume the process is defined until time $s\ge
0$ and proceed by recurrence. 
\begin{enumerate}
\item Define $\tau(\zeta_s,s) := \inf\{t>s\,:\,t\in
\cup_{i,x,y}\cN(i,x,y,\zeta_s(i,x))\cup\cup_{i,j}\cN(i,j)\}$.
\item For $t\in [s,\tau)$ define $\zeta_t=\zeta_s$.
\item If $\tau\in \cN\big(i,x,y,\zeta_s(i,x)\big)$ then set $\zeta_\tau = \zeta_s+
\delta_{(i,y)}-\delta_{(i,x)}$.
\item  If $\tau\in \cN(i,j)$ then set $\zeta_\tau= \zeta_s+ \sum_{x\in\N}
\zeta_s(j,x)\delta_{(i,x)} $.
\end{enumerate}
 
The process is then defined until time $\tau$. Put $s=\tau$ and iterate to
define $\zeta_t$ for all $t\ge 0$. Denote $\zeta^\zeta_t$ the process with
initial state $\zeta$. We leave the reader to prove that $\zeta^\zeta_t$ so
defined is the branching process, that is, a Markov process with rates $b$ and
initial state $\zeta$.

Let $|\zeta|:= \sum_{i,x} \zeta(i,x)$ be the total number of individuals
in~$\zeta$.  Let 
\[
R(\zeta) :=\max\Bigl\{x:\sum_i\zeta(x,i)>0\Bigr\}.
\]
 Let $\tZ^z_t$ be the
process on $\N$ with rates $\tilde q$ and initial position $z\in\N$.

\bl{growth.mt} $ E|\zeta^\zeta_t |\,=\,|\zeta|\,e^{\bar q t}$.  \el
\bpr $ E|\zeta_t |$ satisfies the equation
\be{a524}
\frac{d}{dt} E|\zeta_t | \;=\; \frac{\bar q}{N-1} E\Bigl(\sum_i\sum_{j:j\neq i}\sum_x
\zeta_t(j,x)\Bigr) \; = \;
\frac{\bar q}{N-1}\, (N-1)\, E|\zeta_t| \;=\; \bar q E|\zeta_t|,
\ee
with initial condition $ E|\zeta_0 |=|\zeta|$.
\epr

\bl{a501} Let $g:\N\to\R^+$ be non decreasing. Then
\be{a520}
Eg(R(\zeta^\zeta_t))\;\le\;
E|\zeta^\zeta_t|\,Eg(\tZ
^{R(\zeta)}_t).
\ee
\el
\bpr
Consider the following
partial order on $\mathcal B$:
\be{a521}
\zeta\prec \zeta' \quad \text{ if and only if } \quad \sum_{y\ge x}\zeta(i,y)\le
\sum_{y\ge x}\zeta'(i,y), \quad \text{for all }i, x. 
\ee
The branching process is attractive: the Harris
construction with initial configurations $\zeta\prec \zeta'$ gives
$\zeta^\zeta_t\prec \zeta^{\zeta'}_t$ almost surely; we leave the proof
to the reader. Let $\zeta' := \sum_{i,x}\zeta(i,x)\delta_{(i,R(\zeta))}$ be the 
 configuration having the same number of individuals of type $i$ as $\zeta$ for
 all $i$, but all are located at $r:=R(\zeta)$.  Hence $\zeta\prec\zeta'$ and 
\be{a527}
Eg(R(\zeta^\zeta_t))\; \le\; \sum_{i,x} g(x) E\zeta^\zeta_t(i,x)\;\le\; \sum_x g(x)\sum_i
E\zeta^{\zeta'}_t(i,x),
\ee
because $g$ is non-decreasing. Fix $i$ and $x$ and define
\[
b_t(r,x) := \sum_i E\zeta^{\zeta'}_t(i,x),\quad a_t:=E|\zeta^{\zeta}_t|,\quad
\tilde p_t(r,x):=P(\tZ^{r}_t=x). 
\]
Since $b_t(r,x)$ and $a_t\tilde p_t(r,x)$ satisfy the same Kolmogorov
backwards equations and have the same initial condition, the
right hand side of \eqref{a527} is the same as the right hand side of
\eqref{a520}. 
This can be seen as an application of the one-to-many lemma, see \cite{athreya}.
\epr


\noindent{\bf Harris construction of Fleming-Viot} 
Let $\cN(i,j,x)\subset \cN(i,j)$ be the Poisson process obtained by
independently including each $\tau \in\cN(i,j)$ into $\cN(i,j,x)$ with
probability $q(x,0)/ \bar q$ ($\le1$, by definition of $\bar q$). The processes
$(\cN(i,j,x), i, j\in\{1,\dots,N\}, x\in\N)$ are independent Poisson processes
of rate $q(x,0)/(N-1)$.

Fix $\xi_0=\xi\in\N^{\{1,\dots,N\}}$, assume the process is defined until time
$s\ge 0$ and proceed iteratively from $s=0$ as follows.
\begin{enumerate}
\item Define $\tau(\xi_s,s) = \inf\{t>s\,:\,t\in
\cup_{i,y}\cN(i,\xi_s(i),y,1)\cup\cup_{i,j}\cN(i,j,\xi_s(i))\}$
\item For $t\in [s,\tau)$ define $\xi_t=\xi_s$.
\item  If $\tau\in \cN(i,\xi_s(i),y,1)$, then set $\xi_\tau(i) = y$ and for
$i'\neq i$ set $\xi_\tau(i') = \xi_s(i')$.
\item If $\tau\in \cN(i,j,\xi_s(i))$, then set $\xi_\tau(i)= \xi_s(j)$ and for
$i'\neq i$ set $\xi_\tau(i') = \xi_s(i')$.
\end{enumerate}
The process is then defined until time $\tau$. Put $s=\tau$ and iterate to
define $\xi_t$ for all $t\ge 0$. We leave the reader to prove that
$\xi^\xi_t$ is a Markov process with generator $\cLn$ and initial configuration
$\xi$ and the following lemma.

\bl{a506}
The Fleming-Viot $i$-particle coincides with the position of a branching
$i$-individual at time $t$ if this happens at time zero for all $i$. More
precisely,
\be{a511}
  \zeta_0(i,\xi_0(i))\ge 1\text{ for all }i\text{ implies }\zeta_t(i,\xi_t(i))\ge
1\text{ for all }i,\quad \text{a.s.}.
\ee
\el
\bc{a508} Assume $\zeta_0(i,\xi_0(i))\ge 1$ for
all $i$. Then,
\be{a510}
R(\xi_t) \le R(\zeta_t),\quad\text{a.s.}
\ee
\ec

\section{Galton-Watson estimates}\label{Galton-Watson estimates}

We show now that for $\rho$ small enough the functions $e^{\rho\cdot}$ belong to
the domain of the generator of \gw, that is, the Kolmogorov equations hold for
these functions. The total number of births of the \gw{} process $Z^x_t$ is a
random variable $H^x:=x+\sum_{t>0} (Z^x_t-Z^x_{t-})^+$. Theorem 2 in \cite{NSS}
says that \eqref{a15}-\eqref{def-drift} are equivalent to the existence of a
$\sigma'>0$ such that
\be{hh} 
E(\exp(\sigma' H^1)) \;<\;\infty. 
\ee
Clearly $\sigma'\le \sigma$. Let
%
\be{boldF} 
\bF\;:=\; \Big\{f:\N\cup\{0\}\to\R\,:\, \sum_{\ell\ge
  0}e^{-\rho\ell}|f(\ell)|\,<\,\infty \mbox{ for some } \rho<\sigma'\}.  
\ee
Note that if $f\in\bF$, then there exist $\rho<\sigma'$ and $C>0$  such
that $|f(\ell)|\,\le\, C e^{\rho\ell},\;\ell\ge 0$.  For $f\in\bF$ define the
\gw{} semigroup by
\be{semiproup-gw}
S_tf(x) :=
E(f(Z^x_t))\;<\;\infty,
\ee
because $Z^x_t \le H^x$ for all $t\ge 0$. 
The generator
$Q$ of \gw{} applied on functions $f$ is given by
\be{gen-branching} Qf(x):= \sum_{\ell=-1}^{\infty} xp(\ell+1)
\big( f(x+\ell)-f(x)\big),\qquad x\ge 0,
\ee
if the right hand side is well defined.
\bl{gwke} Under the assumption \eqref{a15}, for $f\in \bF$, $Qf(x)$ is well
defined and the Kolmogorov equations hold:
\be{kegw} \frac{d}{dt}S_t f = QS_tf =
S_tQf.  
\ee
\el
\begin{proof}
Since $|f(x)|\le C\exp(\rho x)$ for all $x \in \N$,
  \be{b11} |Qf(x)| \;\le\; C xe^{\rho x} \Bigl(\sum_{\ell\ge -1}
  p(\ell+1)\,e^{\rho\ell} + 1)\Bigr).  \ee
This shows the first part of the lemma.
Consider $f\in\bF$ and define the local
martingale (see \cite[Section IV-20, pp. 30-37]{RW} )
\[
M^x_t\;:=\; f(Z^x_t) - f(x)- \int_0^t Q f(Z^x_s) ds.
\]
Using \eqref{b11}, for all $s \le t$
\[
|M^1_s|\; \le\; e^\rho+ \exp(\rho H^1) + tC  H^1\exp(\rho H^1) 
\;\le\;  \tilde C \exp(\tilde \rho H^1),
\]
with $\rho < \tilde \rho < \sigma'$.  Hence $ E \sup_{s \in [0,t]} |M^1_s| <
\infty $ and $M^1_t$ is a martingale by dominated convergence.  Since for
$\rho\le \sigma'$, $ E\exp( \rho H^x)=(E \exp( \rho H^1))^x $, the same
reasoning shows that $M^x_t$ is a martingale and
$$
Ef(Z^x_t)= f(x) + E \int Q f(Z^x_s) ds,
$$
which is equivalent to \eqref{kegw} for
$f\in\bF$. 
\end{proof}

The generator of the reflected \gw{} process  $\tZ_t$ reads
\be{def-reflected}
\begin{split}
\tQg f(x):=&\sum_{\ell=-1}^\infty x p(\ell+1) \ind_{\{x+\ell \ge 1\}}
\big(f(x+\ell)-f(x)\big),\qquad x\in \N,
\end{split}
\ee
if the right hand side is well defined. The reflected process can be thought of
as an absorbed process regenerated at position 1 each time it gets
extinct. Since the absorbed process can terminate only when it is at state 1 and
jumps to 0 at rate $p(0)$, the number of regenerations until time $t$ is
dominated by a Poisson random variable $\cN_t$ of mean $tp(0)$ and 
\[
E\big(\exp(\rho \tZ^1_t)\big)\; \le\;  E \exp\Bigl(\rho \sum_{n=1}^{\cN_t}H^1_n\Bigr),
\]
where $H^1_n$ are i.i.d random variables with the same distribution as $H^1$ and
$\cN_t$ is independent of $(H^1_n, n\ge 1)$.  Hence,
\[
E\big(\exp(\rho \tZ^1_t)\big) \;\le\; \exp\big( tp(0)\, C(\rho) \big).
\]
Let $\tSt$ be the semigroup of the reflected \gw{} process. Using the same reasoning as before, we obtain
\bc{tqg}
Any  $f\in\bF$  satisfies the Kolmogorov equations for $\tQg$:
\be{kergw} \frac{d}{dt}\tSt f = \tQg\tSt f =
\tSt\tQg f.  
\ee
\ec

\paragraph{Large deviations}\label{sec-ld}
We study $\tZ _t $, the reflected \gw{} process with
generator $\tQg$ given by \eqref{def-reflected}.
Since $p$ satisfies
\eqref{a15}, for $\rho<\sigma'\le \sigma$,
\be{ld-4}
\Gamma(\rho)\;:=\;
p(0)+\sum_{\ell=1}^\infty p(\ell+1) \ell^2 e^{\rho \ell}\;<\;\infty.
\ee
Recall that $v$ is defined in \eqref{def-drift} and define $\beta$ as
\be{a11}
\beta=\sup\{ \rho > 0 \colon  \rho \Gamma(\rho)\le v \},
\ee
which is well defined thanks to the exponential moment of $p$.

\bl{lem-expon} For any $\rho<
\min\{\beta,\sigma'\}$, and $x\in\N$,
\be{ineq-expon}
E\exp(\rho \tZ^x_t)\;\le\; e^{-\frac{\rho v}{2} t}
e^{\rho x} + te^\rho.
\ee
\el 
\bpr Since $\rho<\sigma'\le \sigma$, the reflected \gw{} generator
\eqref{def-reflected} applied to
$e^{\rho\cdot}$ is well defined and gives
%
\[
\begin{split}
\tQg(e^{\rho \cdot})(x)\;&=\;\sum_{\ell=-1}^\infty x p(\ell+1) e^{\rho x}
\big(e^{\rho \ell}-1\big)-p(0)\ind_{\{x=1\}}\big(1-e^\rho\big)\\
\;&=\; x e^{\rho x} \Big(-\rho v+\sum_{\ell=-1}^\infty p(\ell+1)\big(
e^{\rho \ell}-1-\rho \ell\big)\Big)+p(0)\ind_{\{x=1\}}\big(e^\rho-1\big).
\end{split}
\]
%
Using that for $a\ge 0$, $e^a-(1+a)\le \frac{a^2}{2}e^a$,
\be{expon-2}
\begin{split}\tQg(e^{\rho \cdot})(x)&\;\le\;  \rho x e^{\rho x}\Big(-v+
\frac{\rho}{2}\Gamma(\rho)\Big)+p(0)\ind_{\{x=1\}} e^{\rho}\\
&\;\le\;- \frac{v\rho}{2}e^{\rho x}\,+\,e^{\rho} ,
\end{split}
\ee
using $\rho< \beta$ and $\beta\Gamma(\beta)\le v$. Since $\rho<\sigma'$,
Corollary \ref{tqg} and Gronwall's inequality  give  \reff{ineq-expon}.\epr

We obtain now a Large Deviation estimate.
\bp{prop-LD} 
Let  $T\ge {1 \over 16 v}$ and 
$\delta \ge \max\{1, 4Tp(0)\}$. Then, there is a constant $\kappa$,
independent of $x$, such that
\be{ld-reg}
P\Big( \sup_{s<T}\big( \tZ^x_s-
e^{-v s}x \big)\ge \delta \Big)\;\le\;
\exp\Big(-\frac{\kappa}{T}
  \frac{\delta^2}{\max\{x,\delta\}}\Big).
\ee
\ep
\bpr

Set $z_t^x=e^{-v t}x$ and introduce the process
\be{ld-6}
\begin{split}
\epsilon_t^x\;:=&\;\tZ^x_t-x+v \int_0^t\tZ^x_sds\\
=&\;\big(\tZ _t^x-z_t^x\big)+v \int_0^t\big(\tZ ^x_s-z^x_s\big)ds.
\end{split}
\ee
To stop $\tZ_t^x$ when it crosses $2\max\{x,\delta\}$ define
\be{ld-7}
\tau:=\inf\acc{ t\ge 0:\ \tZ_t^x\ge 2\max\{x,\delta\}}.
\ee
Note that if $\tau<\infty$, 
then $\tZ_\tau^x-z_\tau^x\ge 2\max\{x,\delta\}-x\ge \delta$. Thus,
\be{stop-1}
\acc{\tZ _t^x-z_t^x\ge \delta}\subset 
\acc{\tZ_{t\wedge \tau}^x-z_{t\wedge \tau}^x\ge \delta}.
\ee
For functions $g_1,\,g_2:\R\to\R$ verifying
\[
g_1(t) =g_2(t) +v\int_0^t g_2(s) ds,\qquad v\ge0,
\]
it holds
\[
\sup_{t\le T} |g_1(t) |\le \frac{\delta}{2}\;\Longrightarrow\;
\sup_{t\le T} |g_2(t) |\le \delta.
\]
Hence,
\be{ld-8}
\acc{\sup_{t\le T} \big|\tZ_{t\wedge \tau}^x-z_{t\wedge \tau}^x\big|
\ge \delta}\;\subset\; \acc{\sup_{t\le T} |\epsilon_{t\wedge \tau}^x|\ge
\frac{\delta}{2}}.
\ee
Note that
\[
\acc{\sup_{t\le T} |\epsilon_{t\wedge \tau}^x|\ge
\frac{\delta}{2}}\;=\;
\acc{\sup_{t\le T} \epsilon_{t\wedge \tau}^x\ge \frac{\delta}{2}}
\,\cup\,\acc{\inf_{t\le T} \epsilon_{t\wedge \tau}^x\le -\frac{\delta}{2}}.
\]
The treatment of the two terms on the right hand side of the previous formula is
similar, and we only give the simple argument for the first of them. 
For $\rho<\sigma'$, the
following functional is a local martingale (see \cite[page 66]{EK}).
\be{ld-9}
\cM_t :=\exp\Big({\rho\tZ _t^x}-{\rho x}-\int_0^t \big(e^{-\rho \cdot}
\tQg(e^{\rho \cdot })\big)(\tZ ^x_s)ds\Big).
\ee
Using the bounds of Lemma \ref{gwke} we obtain that $\cM_t$ is in fact a martingale. Observe that
\be{ld-11}
\begin{split}
e^{-\rho x} \tQg(e^{\rho\cdot})(x)
&\;=\;  x \sum_{\ell=-1}^\infty p(\ell+1)\big(e^{\rho \ell}-1\big)+
p(0)\ind_{\{x=1\}} \big( e^{\rho}-1\big)\\
&\;\le\; -\rho v x+\rho\, p(0)+\frac{\rho^2 }{2} \big(
x\Delta(\rho)+p(0)e^{\rho}\big),
\end{split}
\ee
with,
\be{ld-12}
\Delta(\rho):=\frac{2}{\rho^2}
\sum_{\ell=-1}^\infty
p(\ell+1)\big(e^{\rho \ell}-1-\rho \ell\big)\ge 0.
\ee
We have already seen that $\Delta(\rho)\le \Gamma(\rho)$.
Then, we bound the martingale $\cM_t$ as follows.
\be{ld-10}
\begin{split}
\cM_t &\;\ge\; \exp\Big(\rho\big(\tZ _t^x-x\big)-
\big(- \rho v+\frac{\rho^2}{2}\Delta(\rho)\big) \int_0^t
\tZ ^x_sds-\rho p(0)t-\frac{\rho^2}{2}t e^{\rho}\Big)\\
&\;\ge\;\exp\Big(\rho\epsilon_t^x-\rho p(0)t
-\frac{\rho^2}{2}\Gamma(\rho)
\int_0^t \tZ ^x_sds-\frac{\rho^2}{2}t e^{\rho}\Big).
\end{split}
\ee
By stopping the process at $\tau$, and using that $\delta\ge 1$,
we obtain for $t\le T$
\be{ld-14}
\exp\big(\rho\epsilon_{t\wedge \tau}^x\big)\le \cM_{t\wedge \tau}
\exp\big(\rho p(0)T+
\rho^2 \max\{x,\delta\}T\Gamma(\rho)\big).
\ee
Using \reff{stop-1}, \reff{ld-8} and \reff{ld-14}, and the
bound $p(0)T\le \delta/4$, we obtain for any $\rho>0$
\be{ld-15}
\begin{split}
P\big(\sup_{s\le T} \big(\tZ ^x_s
-z^x_s\big)\ge \delta\big)\;&\le\;
P\Big( \sup_{s\le T}  \cM_{s\wedge \tau}\ge \exp\big(
\frac{\rho\delta}{4} -\rho^2  \max\{x,\delta\}T
\Gamma(\rho)\big)\Big)\\
&\le\;\exp\Bigl(-\frac{\rho\delta}{8}+
\rho^2  \max\{x,\delta\}T
\sup_{\rho<\beta} \Gamma(\rho)\Bigr),
\end{split}
\ee 
by Doob's martingale inequality and for $\rho<\beta$. 
Optimize over $0<\rho< \beta$ (recalling that 
$16T\beta \Gamma(\beta)>1$), and choose
$\rho^*$
\be{ld-17}
\rho^*:=\frac{1}{16T} 
\frac{\delta}{\max\{x,\delta\}}\frac{1}{\Gamma(\beta)}<\beta.
\ee
The result follows now from \reff{ld-15} and \reff{ld-17}.
\epr

\section{Bounds for the rightmost \fv-particle}
\label{rightmost-particle}
In this section, we bound small exponential
moments of the rightmost \fv-particle.
We first define a threshold $A$, such that with very small probability,
the rightmost particle's position does not decrease when it is initially
larger than $A$. Define 
\[
\gamma:=\frac{1}{2}\Bigl(1-\exp\Bigl(-\frac{v}{4p(0)}\Bigr)\Bigr) \in (0,1).
\]
Choose
\[
\rho_0 := \frac{\min\{\beta,\sigma',\gamma \kappa p(0)\}}{4}
\]
where $\kappa$ is the constant given by Proposition \ref{prop-LD}. Define
\be{def-amin1}
 A:=\frac{2\kappa p(0)}{\rho_0}>1.
\ee
Define the time and the error $\delta$ entering
in the large deviation estimate of Proposition~\ref{prop-LD} as follows.  For an
arbitrary initial condition $\xi$,
\be{def-amin2}
T:=\frac{1}{4p(0)},\quad\text{and}\quad \delta:=\max\Bigl\{1,
\frac{\rmost(\xi)}{A}\Bigr\},
\ee
recall here that $\rmost(\xi)=\max_{i\le N} \xi(i)$, and set
$V_L(\xi)=\exp(\rho\min( \rmost(\xi),L))$ for $L>A$ which will
be taken to infinity later. We use the notation
$[F(\xi_t)]_0^T:=F(\xi_T)-F(\xi_0)$.
\begin{proof}[Proof of Proposition~\ref{prop-rightmost}]
  We use the construction in Section \ref{sec-construction} to couple the \fv{}
  process $\xi^\xi_t$ and the branching process $\zeta^\zeta_t$ with
  $\zeta=\sum_i\delta_{(i,\xi(i))}$, so that $\zeta(i,\xi(i))=1$ for all
  $i$. Then, by \eqref{a510} $\rmost(\xi_t )\le \rmost(\zeta_t )$ and it is
  sufficient to prove an inequality like \eqref{foster} for
  $\rmost(\zeta_t)$. Notice that for the initial configurations $\xi$ and
  $\zeta$, $R(\xi)=R(\zeta)$. We drop the superscripts
  $\xi$ and $\zeta$ in the remainder of this proof.

 Define the event 
 \be{a80}  \cG=\cG(\xi,T):=\acc{\rmost(\zeta_T)-e^{-v T} \rmost(\xi)\le\
\delta},
\ee
and for a positive real $c$, we define the set
\[
K_c:=\{\xi \colon R(\xi)\le c\}.
\]
 On $K_A^c$, $\delta=\rmost/A<\rmost$, and on $K_A^c\cap \cG$,
\be{lyap-3}
\rmost(\zeta_T)\;\le\; \Bigl(\frac{1}{A}+e^{-v T}\Bigr) \rmost(\xi)\; \le\; (1-\gamma) \rmost(\xi).
\ee
Hence, 
\be{lyap-4} 
\ind_{K_A^c\cap \cG} \cro{ V_L(\zeta_t)}_0^T \;\le\;
V_L(\xi)\pare{e^{-\gamma\rho \rmost(\xi)}-1} \ind_{K_A^c\cap K_L\cap \cG}
\;\le\; -V_L(\xi)\pare{1-e^{-\gamma\rho A}}\ind_{K_A^c
\cap K_L\cap \cG} .  
\ee
Since $A>1$, on $K_A\cap \cG$, $\rmost(\zeta_T)\le Ae^{-vT}+1\le 2A$ so that
\[
\ind_{K_A\cap \cG} \cro{ e^{\rho \rmost(\zeta_t)}}_0^T\le 
e^{2\rho A}\ind_{K_A\cap \cG} .
\]
Thus
\be{lyap-6}
\begin{split}  
\cro{ V_L(\zeta_t)}_0^T&\;\le\; 
-\big(1-e^{-\gamma \rho A}\big) e^{\rho \rmost(\xi)} 
\ind_{K_A^c\cap K_L\cap \cG}\,+\,
e^{2\rho A} \ind_{K_A\cap \cG} \,+\,
\cro{ e^{\rho  \rmost(\zeta)}}_0^T\ind_{\cG^c}\\
&\;\le\; -
\big(1-e^{-\gamma \rho A}\big) V_L(\xi) \ind_{K_A^c\cap K_L}\,+\,
e^{2\rho A} \ind_{K_A}\,+\,2e^{\rho \rmost(\zeta_T)}
\ind_{\cG^c},
\end{split}
\ee
where we used that
\[
\ind_{K_A^c\cap K_L} - \ind_{K_A^c\cap K_L\cap \cG}\le
\ind_{\cG^c}.
\]
Choose $\rho:=\min(\rho_0,\frac{\kappa}{4TA^2})$ and observe that by
Lemma \ref{a501},
\[
 E\cro{e^{2\rho  \rmost(\zeta_T)}} \le E|\zeta_T|\,
E\cro{\exp\Bigl(2\rho \tZ_T^{\rmost(\xi)}\Bigr)} \le Ne^{p(0)T} \left (
e^{-{2\rho v} T}
e^{2\rho R(\xi)} + Te^{2\rho} \right).
\]
by Lemma \ref{growth.mt} for the bound of the first factor and Lemma
\ref{lem-expon} for the bound of the second factor. 
Also, Lemma \ref{prop-LD} implies
\[
 P(\cG^c) \le E|\zeta_T| P\Big( \sup_{s<T}\big( \tZ^{R(\xi)}_s-
e^{-v s}R(\xi) \big)> \delta \Big)\;\le\;
Ne^{p(0)T}(e^{-\frac{\kappa}{TA}}\one_{R(\xi)\le A} + e^{-\frac{\kappa
R(\xi)}{TA^2}}\one_{R(\xi)>A}).
\]
Taking expectation on \reff{lyap-6} we bound the
last term as follows. For constants $C_1,C_2,\tilde C_1,$ and $\tilde C_2$
\be{lyap-5}
\begin{split}
E\cro{e^{\rho \rmost(\zeta_T)}\ind_{\{\cG^c\}}}\;
&\le \; \pare{ P(\cG^c) 
E\cro{e^{2\rho  \rmost(\zeta_T)}}}^{1/2}\\
&\le \; Ne^{p(0)T} \pare{C_1\one_{K_A} + C_2
\exp{\left(-\frac{\kappa R(\xi)}{2TA^2}\right)}\one_{K_A^c} }^{1/2}\;\\
&\le\;  \tilde C_1 N \ind_{K_A}+ \tilde C_2 N  \exp\Bigl(-\frac{\kappa \rmost(\xi)}{4T A^2}\Bigr)
\ind_{K_A^c}.
\end{split}
\ee
Gathering \reff{lyap-5} and \reff{lyap-6} we obtain, for any $L>A$,
\be{foster.copy}
\begin{split}E V_L(\xi^\xi_T) -V_L(\xi)\; & <\;\;
-\,c_1\, V_L(\xi)\ind_{L>R(\xi)>A} + C_1 N \ind_{R(\xi)\le A} + C_2N 
e^{-\rho \tilde c_2R(\xi)}\\
& \le \;\; -\,c_1\, V_L(\xi)\ind_{L>R(\xi)>A} + C N e^{- c_2 R(\xi)}\end{split}  
\ee 
which completes the first part of the proof, inequality \eqref{foster},
as one takes $L$ to infinity in \reff{foster.copy}.

For the second part, take $C>0$ and observe that the set of $\xi$ such
that the right hand
  side of \eqref{foster.copy} is larger than $-C$ is finite. Foster's
criteria, \cite[Theorems 8.6 and 8.13]{Robert(2003)} 
implies that both the chain
  ($\xi_0,\xi_T,\xi_{2T},\cdots)$ and the process $\xi_t $ are ergodic with the
same invariant measure that we call $\lambdan$.

Now, consider again \reff{foster.copy} for a fixed $L$. Note
that $V_L$ is bounded, so that by integrating
\reff{foster.copy} with this invariant measure, and then
taking $L$ to infinity, we obtain \eqref{rightmost-expon}.
\end{proof}

\section{The empirical moments of Fleming-Viot}
\label{moments-fv}

In this section we prove Corollary~\ref{cor-bary}. Introduce the occupation numbers $\eta:\N\times\N^N\to\N$ defined as
\[
\eta(x, \xi) := \sum_{i=1}^N \ind_{\xi(i)=x},
\]
for which we often drop the coordinate $\xi$. Notice that $m(x,\xi) =
\eta(x,\xi)/N$.  

For any integer $k$, define the $k$-th moment of the $N$ particles'
positions as
\[
M_k(\xi) := \sum_{i=1}^N\xi^k(i)=\sum_{x=1}^\infty x^k \eta(x,\xi).
\]
As there are only $N$ particles, $M_k$ is well defined.
Instead of working with the barycenter $M_1/N$, we consider
$\psi:= M_2/M_1$. Note the inequalities
\be{psi-2}
1\le \frac{M_1(\xi)}{N}\le \psi(\xi)\le R(\xi).
\ee
The function $\psi$ is not compactly supported (nor bounded).
Even though $\cLn\psi$ is well defined, we need to use
later that $\int \cLn\psi d\lambda^N=0$. We do so by approximating
$\psi$ by a compactly supported function $\psil$ for which
we have
\be{no-proof}
\int  \cLn\psil d\lambda^N=0,\quad\text{and}\quad
\lim_{L\to\infty} \cLn\psil  =\cLn\psi \quad\text{pointwise}.
\ee
We approximate the unbounded test function $\psi$ by
the following one
\be{testf-bounded}
\psil(\xi)=\frac{M^L_{2}(\xi)}{M_1^L(\xi)},\quad
\text{with}\quad M_k^L(\xi)=\sum_{i=1}^N
\min(\xi^k(i),L^k)=\sum_{x=1}^L x^k \eta(x,\xi)+L^k\sum_{x>L} \eta(x,\xi).
\ee
As $N$ is fixed, $M_k^L=L^k N-\sum_{x=1}^L (L^k-x^k)\eta(x)$,
and has compact support.
It is easy, and we omit the proof, to see that there exist a positive constant
$C$ such that 
\be{no-proof2}
\big| \cLn\psi-\cLn\psil\big|\le 
\big| \cLn\psi\big|+ \big|\cLn\psil\big|\le 
C \psi\le C\,R,
\ee
where we recall that $R(\xi) = \max_i\xi(i)$. We have established in
Proposition~\ref{prop-rightmost} that 
$R(\xi)$ is integrable with respect to
$\lambda^N$, so that \reff{no-proof}
implies that 
\be{no-proof3}
\int  \cLn\psi\, d\lambdan=0.
\ee
The main result of this section is the following.
\bl{lem-psi}
There are positive constants $C_1,C_2$ such that for any integer $N$ large enough,
\be{ineq-psi}
\int \psi\, d\lambdan\;\le\;
C_1+\frac{C_2}{N}\int R^2
d\lambdan.
\ee
\el
\begin{proof}[{\bf Proof of Lemma~\ref{lem-psi}}]
We
  decompose the generator \eqref{generatorfv} into two generators, one governing the refeed part and
  the other the spatial evolution of the particles: $\cLn=\cLn_{\rm
    drift}+\cLn_{\rm refeed}$, which applied to functions depending on
  $\xi$ only through $\eta(\cdot,\xi)$, read
\be{L-jump}
\cLn_{\rm refeed}=p(0)\eta(1)\sum_{x=1}^\infty
\frac{\eta(x)}{N-1} \pare{A_1^-A_{x}^+-\ind},
\quad\text{with}\quad
A_x^\pm(\eta)(y) = \begin{cases}
                     \eta(y) & y\ne x,\\
                     \eta(x) \pm 1 & y=x,
                 \end{cases}
\ee
\be{L-drift}
\cLn_{\rm drift}=\sum_{x=2}^\infty x \eta(x)p(0)(A_x^-A_{x-1}^+-\ind)+
\sum_{x=1}^\infty x \eta(x,\xi) \sum_{i=1}^\infty p(i+1)(A_x^-A_{x+i}^+-\ind).
\ee
It is convenient to introduce a boundary term
\be{boundary-bulk}
B=-\eta(1)p(0)(A_1^-A_{0}^+-\ind)
\quad\text{and call}\quad \cLn_0=\cLn_{\rm drift}-B,
\ee
which applied on $\psi$ yield
\be{Lboundary}
B\psi\;= \;-p(0)\eta(1)\pare{\frac{M_2-M_1}{M_1(M_1-1)}};
\ee
\be{bulk-1}
\begin{split}
\cLn_0  \psi &\; =\;
 \sum_{x=1}^\infty x\eta(x)
\sum_{i=-1}^\infty p(i+1)\pare{\frac{M_2+2ix+i^2}{M_1+i}-
\frac{M_2}{M_1}}\\
&\;=\; \sum_{x=1}^\infty x\eta(x)\acc{
\sum_{i=-1}^\infty ip(i+1)
\pare{\frac{2x M_1-M_2+iM_1}{M_1(M_1+i)}}}\\
&\;=\; -p(0)\frac{M_2-M_1}{M_1-1}+\Big(\sum_{i=1}^\infty
p(i+1)i  \frac{M_1}{M_1+i}\Big)\times
\frac{M_2}{M_1}+\sum_{i=1}^\infty p(i+1)i^2 \frac{M_1}{M_1+i}\\
&\;\le\; -v\psi+p(0)\frac{M_1}{M_1-1}+\sum_{i=1}^\infty p(i+1) i^2
\;\le\; -v\psi+C_0,
\end{split}
\ee
for some positive constant $C_0$. Finally, for the jump term
\be{jump2-1}
\begin{split}
\cLn_{\rm refeed}\psi&\;=\;  p(0)\eta(1)\sum_{x=1}^{\infty}
\frac{\eta(x)}{N-1} \pare{\frac{M_2+x^2-1}{M_1+x-1}-\frac{M_2}{M_1}}\\
&\;=\;  p(0)\eta(1)\sum_{x=1}^{\infty}
\frac{\eta(x)}{N-1} \frac{M_1(x^2-1)-M_2(x-1)}{M_1(M_1-1)}\times
\frac{1}{1+\frac{x}{M_1-1}}.
\end{split}
\ee
If we set $\Delta(x)=1/(1+x)-(1-x)$, for $x\in [0,1]$, then
\be{basic-ineq}
\Delta(x)=\frac{x^2}{1+x},\quad\text{and}\quad
0\le \Delta(x)\le x^2.
\ee
We apply \reff{basic-ineq} to expand the last term in \reff{jump2-1},
with $x/(M_1-1)\le 1$ for $x\le R(\xi)$, and obtain
\be{jump2-2}
\cLn_{\rm refeed}\psi= p(0)\eta(1)\sum_{x=1}^{\infty}
\frac{\eta(x)}{N-1} \frac{M_1(x^2-1)-M_2(x-1)}{M_1(M_1-1)}\times
\pare{1-\frac{x}{M_1-1}+\Delta(\frac{x}{M_1-1})}.
\ee
Note that
\[
\sum_{x=1}^{\infty}\eta(x)\big(M_1x^2-M_2x\big)=0,\quad\text{and}
\quad \sum_{x=1}^{\infty}\eta(x)\big(M_1x^2-M_2x\big)(-x)=
-M_3M_1+(M_2)^2.
\]
Also,
\[
\begin{split}
\sum_{x=1}^{\infty}\frac{\eta(x)}{N-1}\big(M_2-M_1\big)
\pare{1-\frac{x}{M_1-1}}=&\pare{N-\frac{M_1}{M_1-1}}
\frac{\big(M_2-M_1\big)}{N-1}\\
=&\pare{1-\frac{1}{(N-1)(M_1-1)}}\big(M_2-M_1\big)\\
=&\big(M_2-M_1\big)-\frac{M_2-M_1}{(N-1)(M_1-1)}.
\end{split}
\]
Thus
\[
\cLn_{\rm refeed}(\psi)
=-p(0)\frac{\eta(1)}{N-1}
\frac{M_3M_1-(M_2)^2}{M_1(M_1-1)^2}+
p(0)\eta(1)\frac{M_2-M_1}{M_1(M_1-1)}+{\rm Rest},
\]
where
\be{equa-R}
{\rm Rest}\;=\;-\frac{p(0)\eta(1)(M_2-M_1)}{(N-1)(M_1-1)}+
p(0)\eta(1)\sum_{x=1}^{\infty}
\frac{\eta(x)}{N-1}\frac{M_1(x^2-1)-M_2(x-1)}{M_1(M_1-1)}\times
\Delta(\frac{x}{M_1-1}).
\ee
Using that $M_2-M_1\ge 0$, 
\be{jump2-3}
{\rm Rest}\;\le\;
p(0)\frac{\eta(1)}{N-1}\sum_{x= 1}^{\infty}
\eta(x)\times\frac{x^2}{(M_1-1)^2}
\Bigl(\frac{M_1x^2+M_2x}{(M_1-1)^2}\ +\frac{M_2-M_1}
{(M_1-1)^2}\Bigr)
\ee
\be{jump2-4}
\;\le\;
3p(0)\Bigl(\frac{M_1}{M_1-1}\Bigr)^{\!\!4}
\frac{\eta(1)}{N-1}\frac{M_2}{(M_1)^2}
\frac{R^2}{M_1}
\;\le\; 24 p(0) \frac{R^2}{N}.
\ee
Thus, we reach that for $C_0$ independent of $N$ and $L$,
\be{psi2-1}
\cLn(\psi)\;\le\; -v\,\psi\,+\,24\, p(0)\, \frac{R^2}{N}+C_0.
\ee
We now integrate \reff{psi2-1} with respect to the invariant measure,
and use that $\int \cLn \psi d\lambdan=0$ to obtain
for constants $C_1$, and $C_2$ (independent of $N$)
\be{main-psi2}
\int \psi d\lambdan
\le C_1+C_2\frac{\int R^2\,d\lambdan}{N}.
\ee
\end{proof}

\section{Uniform convergence to the Yaglom limit}\label{sec-yaglom}

Define the
generating function of a distribution $\mu$ on $\N$ by
\be{def-gen} G(\mu;z):=\sum_{x\in\N} \mu(x)z^x,\qquad z\in\R,\; |z|<1.  \ee
In this section we show a uniform convergence for $\mu\in K(\alpha)$ of the
generating functions of $\mu T_t$ to the generating function of the \qsd{}
$\nu$. We invoke a key result of Yaglom~\cite{yaglom}. The continuous time
version can be found in Zolotarev \cite{zolotarev}.
\bl{lem-yaglom}[Yaglom 1947, Zolotarev 1957]. 
There is a probability measure $\nu$ such that
\be{lim-yaglom}
\lim_{t\to\infty} G(\delta_1T_t;z)=G(\nu;z),
\ee
and the generating function of $\nu$ is given by
\be{gen-minimal}
G(\nu;z)=1- \exp\Big(-v \int_0^z 
\frac{du}{\sum_{\ell\ge 0}p(\ell) u^\ell-z}\Big),\qquad z\in[0,1).
\ee
\el
The measure $\nu$ is in fact $\nuqs$, the minimal \qsd{}. We do not use the explicit
expression \eqref{gen-minimal} of the generating function of $\nu$; we only use
\eqref{lim-yaglom}. Recall that $\mu T_t$ is the law of $Z_t$ with initial
distribution $\mu$ conditioned on survival until $t$ and that  $K(\alpha)$ is
defined in \reff{def-K}. The next result says that the Yaglom limit
holds uniformly for all initial measures in $K(\alpha)$. 

\bp{prop-yaglom} For any $\alpha >0$ 
\be{yaglom-main} \lim_{N\to\infty}\sup_{ \mu\in K(\alpha)} \big| G(\mu T_t;z)-
G(\nuqs;z)\big|\;=\;0.
\ee
As a consequence, for each $x \in \N$,
\be{unif.Yaglom1}
\lim_{t \to \infty} \sup_{\mu\in K(\alpha)} | \mu T_t(x) - \nuqs(x) |=0.
\ee
\ep

\begin{proof}[Proof of Proposition \ref{prop-yaglom}]
 Recall that $S_t$ is the semigroup of the Galton-Watson process and observe that  for any $\ell\in\N$,
  $G(\delta_\ell S_t;z)=G^\ell(\delta_1S_t;z)$. We set, for simplicity, 
\[
g(z)\;:=\;1-G(\delta_1S_t;z)\;\in\;[0,1],
\]
for $z\in[0,1]$. The following inequalities are useful. For $z\in[0,1]$,
\be{ineq-G}
1-\ell g(z)\;\le\;(1-g(z))^\ell\;\le\; 1-\ell g(z)+\ell^2g^2(z).
\ee
The generating function of $\mu T_t$ reads (the sums run on $\ell\in\N$)
\be{gen-T}
\begin{split}
G(\mu T_t;z)&\;= \;
\frac{G(\mu S_t;z) - G(\mu S_t;0)}{1 - G(\mu S_t;0)}\\
&\;= \; \frac{\sum_{\ell} \mu(\ell ) \big( G(\delta_\ell S_t ;z)
-G(\delta_\ell S_t;0) \big)  }{\sum_{\ell} \mu(\ell )
\big(1-G(\delta_\ell S_t;0)\big)}\\
&\;= \;\frac{\sum_{\ell} \mu(\ell ) \Big( (1-g(z))^\ell -(1-g(0))^\ell \Big)}
{\sum_{\ell} \mu(\ell )\big(1-(1-g(0))^\ell \big)}.
\end{split}
\ee
Also,
\[
1-G(\delta_1T_t;z)\;=\;\frac{1-G(\delta_1S_t;z)}{1-G(\delta_1S_t;0)}
\;=\;\frac{g(z)}{g(0)}.
\] 
We now produce upper and lower bounds for $G(\mu T_t;z)-G(\nuqs;z)$.
We start with the upper bound. Using first \eqref{gen-T} and then \eqref
{ineq-G},
\be{main-matt}
\begin{split}
  G(\mu T_t;z)-G(\nuqs;z)
  &\;= \; \frac{\sum_{\ell} \mu(\ell ) \Big((1-g(z))^\ell -1+
    \big(1-(1-g(0))^\ell \big)(1-G(\nuqs;z))\Big)}
  {\sum_{\ell} \mu(\ell ) \big(1-(1-g(0))^\ell \big)}\\
  &\;\le \; \frac{ \sum_{\ell } \ell \mu(\ell )\Big(-g(z)+\ell
    g^2(z)+g(0)(1-G(\nuqs;z))
    \Big)}{ \sum_{\ell } \ell  \mu(\ell )\Big( g(0)-\ell g^2(0)\Big)}\\
  &\;\le \; \frac{ \sum_{\ell } \ell \mu(\ell
    )\Big((1-G(\nuqs;z))-\frac{g(z)}{g(0)}\Big) + \sum_{\ell }\ell ^2\mu(\ell
    )\frac{g(z)}{g(0)}g(z)}
  {\sum_{\ell } \ell  \mu(\ell )\big(1-\ell g(0)\big)}\\
  &\;\le \; \frac{ G(\delta_1T_t;z)-G(\nuqs;z)+\frac{M_2(\mu)}{M_1(\mu)}
    (1-G(\delta_1T_t;z))g(z)}{1-\frac{M_2(\mu)}{M_1(\mu)} g(0)},
\end{split}
\ee
where $M_k(\mu):=\sum_\ell \ell^k\mu(\ell)$, $k\in\N$. Thus,
\be{conc-upper}
\sup_{\mu\in K(\alpha)} 
G(\mu T_t;z)-G(\nuqs;z)\;\le\;
\frac{|G(\delta_1T_t;z)-G(\nuqs;z)|+ (1-G(\delta_1T_t;z))
g(z)\alpha}{1-\alpha g(0)}.
\ee
Now, for the lower bound, we use similar arguments to 
reach
\be{miss-lower}
\begin{split}
G(\mu T_t;z)-G(\nuqs;z)&\; \ge\;
\frac{ \sum_{\ell } \ell  \mu(\ell )\Big(-g(z)+g(0)(1-G(\nuqs;z))-\ell g^2(0)
(1-G(\nuqs;z))
\Big)}{ \sum_{\ell } \ell  \mu(\ell ) g(0)}\\
&\; \ge\; G(\delta_1T_t;z)-G(\nuqs;z)-\frac{M_2(\mu)}{M_1(\mu)}
g(0)(1-G(\nuqs;z)).
\end{split}
\ee
Thus,
\be{conc-lower}
\inf_{\mu\in K(\alpha)} 
G(\mu T_t;z)-G(\nuqs;z)\;\ge\;
-|G(\delta_1T_t;z)-G(\nuqs;z)|-\alpha g(0)(1-G(\nuqs;z)).
\ee
Since $g(z)$ goes to 0 as the implicit $t$ goes to infinity, both
\reff{conc-upper} and \reff{conc-lower} go to 0. This proves
\eqref{yaglom-main}. The proof of \eqref{unif.Yaglom} follows from
\eqref{yaglom-main} and Lemma \ref{unif.conv} below on convergence of probability measures.  \epr

\bl{unif.conv}
Let $\{\mu_n^\gamma, n \in \N, \gamma \in \Gamma \}$ be a family of probability
measures. Assume that for each $z\in[0,1]$ we have
\be{gener.unif.conv}
\lim_{n\to\infty} 
\sup_{\gamma \in \Gamma} |G(\mu_n^\gamma,z) - G(\nu, z)|=0.
\ee
Then, for each $x \in \N$ we have
\[
\lim_{n\to \infty} \sup_{\gamma \in \Gamma} |\mu_n^\gamma(x) - \nu(x)| = 0.
\]
\el
\begin{proof}
  Let $f=\ind_{\{x\}}$. We consider the one-point compactification of $\N$,
  which we denote $\bar \N=\N\cup\{\infty\}$ and extend $f\colon \bar \N \to \R$
  by $f(\infty)=0$. Since $f$ is continuous function on $\bar \N$, the
  Stone-Weierstrass approximation theorem yields a function $h$, which is a
  linear combination of functions of the form $\{y\mapsto a^y,\ 0 \le a \le 1\}$
  (finite linear combinations of these functions form an algebra that separates
  points and contains the constants), such that for any $\ve>0$, $\sup_{y
    \in \N} |f(y) - h(y)| < \ve.$ Then
\[
\sup_{\gamma} |\mu_n^\gamma(x) - \nu(x)| \;=\;  \sup_{\gamma} |\mu_n^\gamma f -
\nu f| \;\le\; \sup_{\gamma} |\mu_n^\gamma f - \mu_n^\gamma h| \,+\,  \sup_{\gamma}
|\mu_n^\gamma h - \nu h| \,+ \,|\nu h - \nu f|.
\]
The first and the third term on the r.h.s. are smaller than $\ve$ while the
second one goes to zero as $n$ goes to infinity by assumption.
\end{proof}

\section{Closeness of the two semi-groups}\label{sec-semigroups}

In this section we show 
how propagation of chaos implies the closeness 
of $Em(x,\xi^\xi_t)$ and $m(\cdot,\xi)T_t$ uniformly 
in $\xi \in \Lambda^N$.
The arguments are similar to those used in \cite{ferrari-maric,AFG}. The key is
a control of the correlations that we state below. For a signed measure $\mu$ in
$\N$ we will need to work with the $\ell_2$ norm given by $\|\mu\|^2 = \sum_{x\in\N} (\mu(x))^2$.

\bp{prop-semigroups} There exist constants $c$ and $C$ such that, 
\be{l2-convergence}
 \sup_{\xi\in\N^N} 
\|E[m(x,\xi^\xi_t)]-m(\cdot,\xi)T_t\|\; \le \frac{Ce^{ct}}{N}.
\ee
As a consequence,
\be{ptwise-convergence}
 \sup_{\xi\in\N^N} 
|E[m(x,\xi^\xi_t)]-m(\cdot,\xi)T_t(x)|\; \le \frac{Ce^{ct}}{N}.
\qquad x\in\N.
\ee
Furthermore
\be{var0}
 \sup_{\xi\in\N^N} 
E\big[m(x,\xi^\xi_t) - m(\cdot,\xi)T_t\big]^2 \le \frac{Ce^{ct}}{N},\qquad x\in\N.
\ee
\end{proposition}

\begin{proposition}[Proposition 2 of \cite{AFG}]\label{prop-chaos}
For each $t>0$, and any $x,y\in\N$ 
\be{correlations}
\sup_{\xi\in\N^N}\big|E[m(x,\xi^\xi_t)m(y,\xi^\xi_t)]
-E[m(y,\xi^\xi_t)]\,E[m(x,\xi^\xi_t)]\big|\;\le\; \frac{2p(0)e^{2p(0)t}}{N}.
\ee
\end{proposition}

The paper  \cite{AFG} proves this proposition for processes
with bounded rates, but the extension to our case is straightforward.

\begin{proof}[Proof of Proposition \ref{prop-semigroups}]
Fix $\xi\in\N^N$ and introduce the simplifying notations
\be{symbol-1} 
u(t,x):=Em(x,\xi^\xi_t)\quad
\text{and}\quad v(t,x):=m(\cdot,\xi)T_t(x).
\ee
Define $\delta(t,x)=u(t,x)-v(t,x)$. We want to show that for any $t>0$,
\be{norm-decay}
\frac{\partial}{\partial t} \|\delta(t) \|^2\;\le\;
\frac52 \|\delta(t) \|^2+\frac{4 p(0)e^{2p(0) t}}{N}.
\ee
Recall the definition  \eqref{rates} of the rates $q$  and the evolution
equations satisfied by $v(t,x)$ and $u(t,x)$:
\be{dynamic-v}
\frac{\partial}{\partial t} v(t,x)=\sum_{z\not= x,z>0}q(z,x)
v(t,z)-\Bigl(\sum_{z\not= x}q(x,z)\Bigr)v(t,x)+p(0) v(t,1)v(t,x),
\ee
\be{a82} 
\frac{\partial}{\partial t} u(t,x)=\sum_{z\not= x,z>0}q(z,x)
u(t,z)-\Bigl(\sum_{z\not= x}q(x,z)\Bigr)u(t,x)+p(0) u(t,1)u(t,x)+
W(\xi;t,x).
\ee
Here,
\be{def-R}
W(\xi;t,x)=p(0) \Bigl(\frac{N}{N-1}E[m(x,\xi^\xi_t)m(1,\xi^\xi_t)]
-E[m(1,\xi^\xi_t)]\,E[m(x,\xi^\xi_t)]\Bigr).
\ee
Proposition~\ref{prop-chaos} implies that
\be{step-31}
\sup_{\xi}|W(\xi;t,x)|\le \frac{ 2p(0)e^{2p(0)t}}{N}.
\ee
Observe two simple
facts. First, set $D=\{(x,z):\ x\ge 1,z\ge 1,\ x\not= z\}$,
and for any function $f:\N\to \R$
\be{obs-1}
\sum_{(x,z)\in D} \big(q(x,z)+q(z,x)\big) f^2(x)-
2\sum_{(x,z)\in D}q(x,z) f(x)f(z)=\sum_{(x,z)\in D}q(z,x)
(f(x)-f(z))^2.
\ee
The second observation is specific to our rates. For $x>0$
\be{obs-2}
\sum_{z\not= x}q(z,x) \le \sum_{z\not= x}q(x,z)+p(0).
\ee
Observation \reff{obs-1} is obvious and we omit its proof.
Observation \reff{obs-2} is done in details.
\be{step-2}
\begin{split}
\sum_{z\not= x}q(z,x)&\;=\; \sum_{z\ge 0,z\not= x} z p(x-z+1)
=
x\sum_{z\ge 0,z\not= x}p(x-z+1)+
\sum_{z\ge 0,z\not= x}(z-x)p(x-z+1)\\
&\;=\;x\big(p(0)+p(1)+\dots+p(x+1)\big)+\big(p(0)-p(2)-\dots-xp(x+1)\big)\\
&\;\le\;  x\sum_{i\ge 0} p(i)+p(0)=\sum_{z\not= x}q(x,z)+p(0).
\end{split}
\ee
Now, we have
\be{step-4}
\begin{split}
\sum_{x>0} \delta(t,x)\frac{\partial}{\partial t} \delta(t,x)
&\;=\;
\sum_{(x,z)\in D}\big( q(z,x) \delta(t,x)\delta(t,z)-
q(x,z) \delta^2(t,x)\big) \\
&\;+\;p(0)\sum_{x>0}  \big( u(t,x)u(t,1)-v(t,x)v(t,1)\big)\delta(t,x)
+\sum_{x>0} \delta(t,x)W(\xi;t,x).
\end{split}
\ee
Let us deal with each term of the right hand side of \reff{step-4}. For the first term we use \reff{obs-1} and \reff{obs-2}.
\be{term-1}
\begin{split}
&\sum_{(x,z)\in D}\big( q(z,x) \delta(t,x)\delta(t,z)-   
q(x,z) \delta^2(t,x)\big)\\
&\qquad\qquad\le\;
\sum_{(x,z)\in D} q(z,x) \delta(t,x)\delta(t,z)-\frac{1}{2}\sum_{x>0}\big( \sum_{z\not= x} q(x,z)+
\sum_{z\not= x} q(z,x) -p(0)\big) \delta^2(t,x)\\
&\qquad\qquad\le\;-\frac{1}{2}\sum_{(x,z)\in D}q(z,x)
(\delta(t,x)-\delta(t,z))^2+\frac{p(0)}{2}\|\delta(t)\|^2\\
&\qquad\qquad\le\;\frac{p(0)}{2}\|\delta(t)\|^2.
\end{split}
\ee
To deal with the second term, first note that
\[
\sup_{x>0} \big| \delta(t,x)\big|\;\le\;
\sqrt{ \sum_{x>0} \delta^2(t,x)}\;=\; \|\delta(t)\|.
\]
Then,
\be{term-3}
\begin{split}
\sum_{x>0}&\big(u(t,x)u(t,1)-v(t,x)v(t,1)\big)\delta(t,x)
\le \sum_{x>0} \big(\delta(t,x) u(t,1)+v(t,x) \delta(t,1)\big) \delta(t,x)\\
\le & \sum_{x>0} \delta^2(t,x)+
|\delta(t,1)| \sup_{x>0} \big| \delta(t,x)\big|\sum_{x>0} v(t,x)
\le 2\|\delta(t)\|^2.
\end{split}
\ee
For the last term, we have
\be{term-4}
|\sum_{x>0} \delta(t,x)W(\xi;t,x)|\le
\sup_{x>0}|W(\xi;t,x)|\times \sum_{x>0} |\delta(t,x)| 
\le 2\sup_{x>0} |W(\xi;t,x)|.
\ee
Thus, we obtain \reff{norm-decay}. Gronwall's inequality allows to conclude.
\end{proof}

\paragraph{\bf Acknowledgements}
We would like to thank Elie Aidekon for valuable discussions.
A.A.'s mission at Buenos Aires was supported by MathAmSud, 
and he acknowledges partial support of ANR-2010-BLAN-0108.

\obeylines
\parskip 0pt
Amine Asselah                    
LAMA, Bat. P3/4,
Universit\'e Paris-Est Cr\'eteil,
61 Av.\/ General de Gaulle,
94010 Cr\'eteil Cedex, France
{\tt amine.asselah@univ-paris12.fr}

\vskip 2mm

Pablo A. Ferrari and Pablo Groisman
Departamento de Matem\'atica
Facultad de Ciencias Exactas y Naturales
Universidad de Buenos Aires
Pabell\'on 1, Ciudad Universitaria
1428 Buenos Aires 
Argentina
{\tt pferrari@dm.uba.ar, pgroisma@dm.uba.ar}
\vskip 2mm

Matthieu Jonckheere
Instituto de Investigaciones Matem\'aticas Luis Santal\'o
 Pabell\'on 1, Ciudad Universitaria
1428 Buenos Aires 
Argentina
{\tt mjonckhe@dm.uba.ar}


\begin{thebibliography}{99}

\bibitem{AFG}Asselah, A., Ferrari, P. A., Groisman, P.  Quasi-stationary
  distributions and Fleming-Viot processes in finite spaces.  {\it
    J. Appl. Probab. \bf 48}, (2011), 2: 322-332.

\bibitem{athreya}
Athreya, K. B., Ney, P.  {\it Branching processes}. Springer-Verlag, Berlin, New
York,  1972.

\bibitem{berard-gouere} B\'erard, J.,  Gou\'er\'e, J.B.   Brunet-Derrida
behavior of branching-selection particles systems on the line.  {\it Comm. Math. Phys.
\bf  298}  (2010),  no. 2, 323--342.

\bibitem{berestycki} Berestycki J., Berestycki N., Schweinsberg J.,
  The genealogy of branching Brownian motion with absorption ,
arXiv:1001.2337v2.

\bibitem{bieniek}  Bieniek, M., Burdzy, K., Finch, S.   Non-extinction
of a Fleming-Viot particle model.  {\it Probability Theory and Related
Fields}.Volume 153 , Numbers 1-2, 293-332 .

\bibitem{bieniek-burdzy}  Bieniek, M., Burdzy, K., Soumik, P.  
Extinction of Fleming-Viot-type particles systems with strong
drift.  arXiv:1111.0078v1 

\bibitem{brunet-derrida1}  Brunet, E.,  Derrida,
B.   Effect of microscopic noise on front propagation. 
 {\it J. Statist. Phys. \bf  103}  (2001),  no. 1-2, 269--282.
		
\bibitem{brunet-derrida2} Brunet, E.,  Derrida, B.
   Shift in the velocity of a front due to a cutoff. 
 {\it Phys. Rev. E (3) \bf  56}  (1997),  no. 3, part A, 2597--2604.

\bibitem{BDMM1}
Brunet, E., Derrida, B., Mueller, A.H., Munier S., 
Noisy traveling waves: effect of selection on genealogies 
{\it Europhys. Lett. \bf  76} (2006), no. 1, 1–7.


\bibitem{BDMM2} Brunet, E., Derrida, B.,  Mueller, A. H.,  Munier, S. 
Effect of selection on ancestry: an exactly soluble case and its
phenomenological generalization.  
{\it Phys. Rev. E (3) \bf  76}  (2007),  no. 4, 041104, 20pp.

\bibitem{burdzy1} Burdzy, K., Holyst, R., Ingerman, D., March, P. 
  Configurational transition in a Fleming-Viot-type 
model and probabilistic interpretation of Laplacian eigenfunctions 
{\it J. Phys. A: Math. Gen. \bf 29} (1996) 2633--2642.

\bibitem{burdzy2}
Burdzy, K.,  Holyst, R., March, P.   A Fleming-Viot particle
representation of Dirichlet Laplacian.  {\it Comm. Math. Phys. \bf214} (2000), 679--703.

\bibitem{cavender}Cavender, J.A.   Quasi-stationary distributions of
birth-and-death processes.  {\it Adv. Appl. Probab. \bf10} (1978), no. 3, 570--586.


\bibitem{durret}
Durrett R., Remenik D., Brunet-Derrida particles systems, free boundary
problems and Wiener–Hopf equations. {\it Ann. Probab. \bf39} (2011) no. 6
2043-2078, 

\bibitem{EK} Ethier, S.N., Kurtz, T.G.,  {\em Markov processes. Characterization
and convergence.} Wiley Series in Probability and Mathematical Statistics:
Probability and Mathematical Statistics. John Wiley \& Sons, Inc., New York
(1986).

\bibitem{ferrari-kesten} Ferrari, P.A., Kesten, H., Martinez, S., Picco, P.
Existence of quasi-stationary distributions. 
A renewal dynamical approach.  {\it Ann. Probab. \bf 23} (1995) no. 2 501-521. 

\bibitem{ferrari-maric} Ferrari, P.A., Maric, N. Quasi-stationary
distributions and Fleming-Viot processes in countable spaces
{\it Electron. J. Probab. \bf 12} (2007), no. 24, 684–702.

\bibitem{fv} Fleming, W.H., Viot, M. Some measure-valued Markov processes in
population genetics theory. {\it Indiana Univ. Math. J.  \bf 28} (1979),
no. 5, 817--843.

\bibitem{GK1} Grigorescu, I., Kang, M.  Hydrodynamic limit for a
Fleming-Viot type system. 
{\it Stochastic Process. Appl. \bf 110} (2004),  no. 1,
111--143.


\bibitem{GK} Grigorescu, I., Kang, M. (2011) Immortal particle for a
catalytic branching process, {\it Probability Theory and Related Fields},
Volume 153 , Numbers 1-2 ,333-361.


\bibitem{maillard}
Maillard, P.  Branching Brownian motion with selection of the $N$
right-most particles: An approximate model. arXiv:1112.0266v2.


\bibitem{NSS} Nakayama, M. K., Shahabuddin, P. and Sigman, K., On Finite
  Exponential Moments for Branching Processes and Busy Periods for Queues,
  {\it Journal of Applied Probability, \bf 41}, 2004.

\bibitem{Robert(2003)}
Robert,P.
\newblock \emph{Stochastic networks and queues}, Applications of
Mathematics, 52. \newblock Stochastic Modelling and Applied Probability.
\newblock Springer-Verlag (New York), 2003.


\bibitem{RW}
Rogers, L.C.G.,  Williams, D,
{\it Diffusions, Markov processes and martingales, 1: Foundations }, second
ed, Wiley \& Sons Ltd., Chichester, 1994. 

\bibitem{seneta-veres} Seneta, E., Vere-Jones, D. 
On quasi-stationary distributions in discrete-time Markov 
chains with a denumerable infinity of states.
{\it J. Appl. Probability \bf 3} (1966) 403--434.

\bibitem{villemonais} Villemonais, D., Interacting particle systems and
Yaglom limit approximation of diffusions with unbounded drift, {\it Electronic
Journal of Probability, \bf 16} (2011), 1663--1692. 

\bibitem{yaglom} Yaglom, A. M. Certain limit theorems of the theory of
branching random processes. {\it Doklady Akad. Nauk SSSR (N.S.) \bf 56}, (1947).
795--798. 

\bibitem{zolotarev} 
Zolotarev, V. M., More exact statements of several theorems in the theory of
branching processes. {\it Theory Probab. Appl. \bf 2} (1957) no. 3, 245-253.


\end{thebibliography}
\end{document}